\documentclass[12pt]{amsart}
\usepackage{amsmath,amssymb}
\usepackage{amsfonts}
\usepackage{amsthm}
\usepackage{latexsym}
\usepackage{graphicx}

\def\p{\partial}

\def\R{\mathbb{R}}

\def\vv<#1>{\langle#1\rangle}

\def\XXint#1#2{\setbox0=\hbox{$#1{#2}{\int}$}{#2}\kern-.5\wd0 }

\def\XXint#1#2#3{{\setbox0=\hbox{$#1{#2#3}{\int}$}
     \vcenter{\hbox{$#2#3$}}\kern-.5\wd0}}

\def\vv<#1>{{\left\langle#1\right\rangle}}

\newtheorem{thm}{Theorem}[section]

\newtheorem{lem}{Lemma}[section]

\newtheorem{cor}{Corollary}[section]
\theoremstyle{definition}

\theoremstyle{remark}

\newtheorem{rem}{Remark}[section]

\numberwithin{equation}{section}

\begin{document}
\title{Trace and inverse trace of Steklov eigenvalues II}

\author{Yongjie Shi}
\address{Department of Mathematics, Shantou University, Shantou, Guangdong, 515063, China}
\email{yjshi@stu.edu.cn}
\author{Chengjie Yu$^1$}
\address{Department of Mathematics, Shantou University, Shantou, Guangdong, 515063, China}
\email{cjyu@stu.edu.cn}
\thanks{$^1$Research partially supported by a supporting project from the Department of Education of Guangdong Province with contract no. Yq2013073, the Yangfan project from Guangdong Province and NSFC 11571215.}
\renewcommand{\subjclassname}{%
  \textup{2010} Mathematics Subject Classification}
\date{}
\begin{abstract}
This is a continuation of our previous work [J. Differential Equations, 261 (2016), no. 3, 2026--2040.] on the trace and inverse trace of Steklov eigenvalues. More new inequalities for the trace and inverse trace of Steklov eigenvalues are obtained.
\end{abstract}
\maketitle\markboth{Shi \& Yu}{Trace and inverse trace II}
\section{Introduction}
For a compact oriented Riemannian manifold $(M^n,g)$ with nonempty boundary, the Dirichlet-to-Neumann map $L^{(p)}:A^p(\p M)\to A^p(\p M)$ for differential $p$-forms maps $\omega\in A^p(\p M)$ to $i_\nu d\hat\omega$ where $\hat\omega\in A^p(M)$ is the tangential harmonic extension of $\omega$ and $\nu$ is the outward unit normal vector on $\p M$. This new notion of Dirichlet-to-Neumann map was recently introduced by Raulot and Savo \cite{RS2}. When $p=0$, $L^{(0)}$ coincides the classical Dirichlet-to-Neumann map or Steklov operator essentially introduced by Steklov \cite{St}. The same as the Steklov operator (see \cite{Ta} for example), $L^{(p)}$ was proved to be a nonnegative self-adjoint first order elliptic pseudo-differential operator by Raulot-Savo \cite{RS2}. Hence, the eigenvalues of $L^{(p)}$ is discrete, and can be listed in ascending order counting multiplicities as follows:
\begin{equation}
0\leq \sigma_1^{(p)}\leq \sigma_2^{(p)}\leq\cdots\leq\sigma_k^{(p)}\leq\cdots.
\end{equation}
They are called Steklov eigenvalues for differential $p$-forms on $M$.

It is clear that $\sigma_1^{(0)}=0$ with constant function as the eigenfunction and $\sigma_2^{(0)}>0$. However, this is not true for $p\geq 1$. Indeed, it is not hard to see that (see \cite{RS2})
\begin{equation}
\ker L^{(p)}=\mathcal H_N^p(M).
\end{equation}
Here
\begin{equation}
\mathcal H_N^p(M)=\{\omega\in A^p(M)\ |\ d\omega=\delta\omega=0\ \mbox{and } i_\nu\omega=0\}.
\end{equation}
By Hodge theory on compact manifolds with nonempty boundary (see \cite{S}),
\begin{equation}
\dim\mathcal H_N^p(M)=b_p
\end{equation}
where $b_p$ is the $p$-th Betti number of $M$. Hence, the multiplicity of the eigenvalue $0$ for $L^{(p)}$ is the same as the $p$-th Betti number of $M$.

There has been many works on Steklov eigenvalues (see for example \cite{Br,CEG,Di,E,FS1,FS2,GP1,GP2,GP3,HPS,IM,K,RS1,RS2,RS3,W,WX}) and Dirichlet-to-Neumman map since their importance in mathematical physics (see \cite{Ku}) and applied mathematics (see \cite{Ul}). It is really hard to give a complete list for works on estimates of Steklov eigenvalues. One can consult the survey \cite{GP} for recent progresses.

In \cite{HPS}, Hersch-Payne-Schiffer proved the following interesting inequality for bounded simply connected planar domain $\Omega$ by using harmonic conjugate:
\begin{equation}\label{eqn-HPS}
\sigma_{p+1}^{(0)}\sigma_{q+1}^{(0)}L(\p\Omega)^2\leq \left\{\begin{array}{ll}(p+q)^2\pi^2& p+q\ \mbox{is even}\\
(p+q-1)^2\pi^2& p+q\ \mbox{is odd.}\end{array}\right.
\end{equation}
Here $L(\p\Omega)$ means the length of $\p \Omega$. This result was generalized by Girouard and Polterovich \cite{GP1} to general surfaces:
\begin{equation}\label{eqn-GP}
\sigma_{p+1}^{(0)}\sigma_{q+1}^{(0)}L(\p M)^2\leq \left\{\begin{array}{ll}(\gamma+k)^2(p+q)^2\pi^2& p+q\ \mbox{is even}\\
(\gamma+k)^2(p+q-1)^2\pi^2& p+q\ \mbox{is odd.}\end{array}\right.
\end{equation}
Here $M$ is a compact surface with genus $\gamma$ and $k$ boundary components. Note that, by setting $p=q$ in \eqref{eqn-HPS} and \eqref{eqn-GP}, one can obtain estimates for Steklov eigenvalues that generalized the classical result of Weinstock \cite{W} and a result of Fraser and Schoen \cite{FS1} respectively.

In \cite{YY2}, Liangwei Yang and the second author generalized \eqref{eqn-HPS} to higher dimensional case by applying the trick of harmonic conjugate introduced by Hersch-Payne-Schiffer \cite{HPS} to the new setting of Steklov eigenvalues for differential forms introduced by Raulot-Savo \cite{RS2}. The result is as follows:
\begin{equation}\label{eqn-YY}
\sigma_{1+p}^{(0)}\sigma_{b_{n-2}+q}^{(n-2)}\leq\lambda_{b_{n-1}+p+q}(\p M).
\end{equation}
Here $M$ is of dimension $n$ and $\lambda_{k}(\p M)$ means the $k$-th eigenvalues of $\p M$ for the Laplacian operator. This result produces new estimates even in the case of surfaces. Indeed, in \cite{Ka}, Karpukhin proved the following inequality by using \eqref{eqn-YY}:
\begin{equation}\label{eqn-K}
\sigma_{p+1}^{(0)}\sigma_{q+1}^{(0)}L(\p M)^2\leq \left\{\begin{array}{ll}(p+q+2\gamma+2k-2)^2\pi^2& p+q\ \mbox{is even}\\
(p+q+2\gamma+2k-1)^2\pi^2& p+q\ \mbox{is odd.}\end{array}\right.
\end{equation}
Here $M$ is a compact oriented surface with genus $\gamma$ and $k$ boundary components. It is clear that  inequality \eqref{eqn-K} is sharper than \eqref{eqn-GP}.

In \cite{HPS}, Hersch, Payne and Schiffer also obtained some sharp estimates on the inverse trace of Steklov eigenvalues on bounded simply connected planar domain $\Omega$ with smooth boundary. Their result is:
\begin{equation}\label{eqn-HPS-trace}
\sum_{i=1}^{2n}\frac{1}{\sigma_{1+i}^{(0)}}\geq \frac{L(\p\Omega)}{\pi}\sum_{i=1}^{n}\frac{1}{i}
\end{equation}
for any positive integer $n$. It is not hard to see that the inequality is sharp on the unit disk. This estimate was generalized to general surfaces by the authors in \cite{SY}.

The equality \eqref{eqn-HPS-trace} can be reformulated in majorization relations. Let $x=(x_1,x_2,\cdots,x_n)\in\R^n$. Rearrange its components in descending order as $x_{[1]}\geq x_{[2]}\geq\cdots\geq x_{[n]}$. Let $y\in \R^n$ be another vector. We say that $x$ is weakly majorized by $y$, denoted by $x\prec_wy$, if
\begin{equation}
\sum_{i=1}^mx_{[i]}\leq \sum_{i=1}^my_{[i]}.
\end{equation}
for any $m=1,2,\cdots,n$. Furthermore, we say that $x$ is majorized by $y$, denoted as $x\prec y$, if $x\prec_wy$ and $\sum_{i=1}^nx_i=\sum_{i=1}^ny_i$. Now, using the majorization relations, \eqref{eqn-HPS-trace} can be reformulated as
\begin{equation}\label{eqn-HPS-majorized}
\frac{L(\p\Omega)}{\pi}\left(1,\frac12,\cdots,\frac1n\right)\prec_w\left(\frac{1}{\sigma_2^{(0)}}+\frac{1}{\sigma_3^{(0)}},\frac{1}{\sigma_4^{(0)}}+\frac{1}{\sigma_5^{(0)}},\cdots,\frac{1}{\sigma_{2n}^{(0)}}+\frac{1}{\sigma_{2n+1}^{(0)}}\right).
\end{equation}

The following two basic majorization principles are useful for producing new inequalities:
\begin{enumerate}
\item if $x\prec_wy$ and $f$ is an increasing convex function, then $$(f(x_1),f(x_2),\cdots,f(x_n))\prec_w(f(y_1),f(y_2),\cdots,f(y_n));$$
\item if $x\prec y$ and $f$ is a convex function, then $$(f(x_1),f(x_2),\cdots,f(x_n))\prec_w(f(y_1),f(y_2),\cdots,f(y_n)).$$
\end{enumerate}
By applying the basic majorization principle with $f(t)=t^2$ to \eqref{eqn-HPS-majorized}, one has
\begin{equation}\label{eqn-inverse-trace-2}
\left(\frac{1}{\sigma_{2}^{(0)}}\right)^2+\left(\frac{1}{\sigma_{3}^{(0)}}\right)^2+\cdots+\left(\frac{1}{\sigma_{{2n+1}}^{(0)}}\right)^2\geq \frac{L(\p\Omega)^2}{2\pi^2}\left(1+\frac{1}{2^2}+\frac{1}{3^2}+\cdots+\frac1{n^2}\right).
\end{equation}
As mentioned in \cite{SY}, a weaker version of this inequality can be found in \cite{Di}. In fact, the general estimates in \cite{SY} can also be obtained in this way.

In \cite{HPS}, the authors also posed the following interesting question: is the following inequality true for bounded simply connected planar domain $\Omega$ with smooth boundary:
\begin{equation}
\frac{1}{\sigma_{2}^{(0)}\sigma_3^{(0)}}+\frac{1}{\sigma_{3}^{(0)}\sigma_4^{(0)}}+\cdots+\frac{1}{\sigma_{2n}^{(0)}\sigma_{2n+1}^{(0)}}\geq \frac{L(\p\Omega)^2}{4\pi^2}\left(1+\frac{1}{2^2}+\frac{1}{3^2}+\cdots+\frac1{n^2}\right)?
\end{equation}
According to the knowledge of the authors, no answer have been found. \eqref{eqn-inverse-trace-2} can be viewed also a weaker version of the inequality. The question of Hersch-Payne-Schiffer is the motivation of the study of this paper.

Furthermore, note that the inequalities \eqref{eqn-YY} and results in \cite{SY} have a similar feature. The left hand sides of the inequalities depend on the geometry of $M$ by definition while the right hand sides of the inequalities depend only on the intrinsic geometry of $\p M$. This may in some sense relate to the interesting problem of determining the geometry of $M$ from the Steklov spectrum or the Steklov operator. In this paper, we obtain new inequalities in a similar feature on the trace or inverse trace of Steklov eigenvalues by combining the tricks in \cite{YY} and \cite{YY2}.

The first main result of this paper is the following inequality mixing up trace and inverse trace of Steklov eigenvalues with weights.
\begin{thm}\label{thm-1}
Let $(M^n,g)$ be a compact oriented Riemannian manifold with nonempty boundary.  Then, for any positive integer $r,s$ and $m$,
\begin{equation}\label{eqn-thm-1}
\begin{split}
\left(\sum_{i=1}^m\left(a_i\sigma_{b_{n-2}+s+i-1}^{(n-2)}\right)^\frac{1}{p}\right)\left(\sum_{i=1}^{m}\left(\frac{c_i}{\sigma_{r+i}^{(0)}}\right)^{\frac{q}{p}}\right)^{-\frac{1}{q}}\leq \left(\sum_{i=1}^m\left(\frac{a_i\lambda_{b_{n-1}+r+s+i-1}}{c_i}\right)^\frac{q^*}{p}\right)^{\frac{1}{q^*}}
\end{split}
\end{equation}
where $q\geq p\geq 1$, $q>1$,$\frac{1}{q^*}+\frac{1}{q}=1$,
\begin{equation}
a_1\geq a_2\geq\cdots\geq a_m\geq0\ \mbox{and}
\end{equation}
\begin{equation}
c_1\geq c_2\geq\cdots\geq c_m>0.
\end{equation}
Here $\lambda_k$ means the $k$-th eigenvalue of the Laplacian operator on $\p M$.
\end{thm}
When $m=1$, Theorem \ref{thm-1} give us \eqref{eqn-YY}. The weights $a_1,a_2,\cdots a_m$ and $c_1,c_2,\cdots,c_m$ can be used to make the inequality sharper. For example, when $M$ is a simply connected surface, we can choose suitable weights to make \eqref{eqn-thm-1} sharp on the unit disk.
\begin{cor}
Let $M^2$ be a compact oriented simply connected surface. Then, for any positive integer $n$,
\begin{equation}\label{eqn-cor-1}
\sum_{i=1}^{n}\frac{\sigma_{2i}^{(0)}+\sigma_{2i+1}^{(0)}}{i^3}\leq \frac{4\sqrt 2\pi^2}{L(\p M)^2} \left(\sum_{i=1}^{2n}\left(\frac{1}{\sigma_{1+i}^{(0)}}\right)^2\right)^\frac{1}{2}\left(\sum_{i=1}^{n}\frac{1}{i^2}\right)^\frac{1}{2}.
\end{equation}
The equality holds when $M$ is a disk. Letting $n\to\infty$, one have
\begin{equation}
\sum_{i=1}^{\infty}\frac{\sigma_{2i}^{(0)}+\sigma_{2i+1}^{(0)}}{i^3}\leq\frac{4\sqrt 3\pi^3}{3L(\p M)^2}\left(\sum_{i=1}^{\infty}\left(\frac{1}{\sigma_{1+i}^{(0)}}\right)^2\right)^\frac{1}{2}.
\end{equation}
\end{cor}
\begin{proof}
Note that, in this case, $b_0=1$ and $b_1=0$. Let $p=1$, $q=2$, $r=s=1$, $m=2n$
\begin{equation}
a_{2i-1}=a_{2i}=\left(\frac{2i\pi}{L(\p M)}\right)^{-3},
\end{equation}
and $c_{2i-1}=c_{2i}=1$ for $i=1,2,\cdots,n$ in \eqref{eqn-thm-1}. Then, the conclusion follows by noting that
\begin{equation}
\lambda_{2i}(\p M)=\lambda_{2i+1}(\p M)=\left(\frac{2i\pi}{L(\p M)}\right)^2
\end{equation}
for $i=1,2,\cdots.$
\end{proof}

One can produce many inequalities of a similar form with \eqref{eqn-cor-1} which is sharp on the unit disk by choosing suitable weights.

The second main result of this paper is an inequality mixing up different types of inverse traces for Steklov eigenvalues as follows.

\begin{thm}\label{thm-2}
Let $(M^n,g)$ be a compact Riemannian manifold with nonempty boundary.  Then, for any positive integer $r,s,m$ and $k=1,2,\cdots,m$,
\begin{equation}\label{eqn-thm-2}
\begin{split}
&\sum_{1\leq i_1<i_2<\cdots<i_k\leq m}\left(\frac{1}{\sigma_{b_{n-2}+s+i_1-1}^{(n-2)}\sigma_{b_{n-2}+s+i_2-1}^{(n-2)}\cdots\sigma_{b_{n-2}+s+i_k-1}^{(n-2)}}\right)^p+\mu C_{m-1}^{k-1}\sum_{i=1}^{m}\left(\frac{1}{\sigma_{r+i}^{(0)}}\right)^q\\
\geq&\frac{kp+q}{pq}p^\frac{q}{kp+q}q^\frac{kp}{kp+q}\mu^\frac{kp}{kp+q}\times\\
&\sum_{1\leq i_1<i_2<\cdots<i_k\leq m}\left(\frac{1}{\lambda_{b_{n-1}+r+s+i_1-1}\lambda_{b_{n-1}+r+s+i_2-1}\cdots\lambda_{b_{n-1}+r+s+i_k-1}}\right)^\frac{pq}{kp+q}\\
\end{split}
\end{equation}
where $p>0$, $q\geq 1$ and $\mu>0$.
\end{thm}
When $p=q$ and $k=\mu=1$ in \eqref{eqn-thm-2}, we have
\begin{equation}\label{eqn-power-q}
\sum_{i=1}^{m}\left(\frac{1}{\sigma_{r+i}^{(0)}}\right)^q+\sum_{i=1}^{m}\left(\frac{1}{\sigma_{b_{n-2}+r+i}^{(n-2)}}\right)^q\geq 2\sum_{i=1}^{m}\left(\frac{1}{\lambda_{b_{n-1}+r+s+i-1}}\right)^\frac{q}{2}.
\end{equation}
This is a special case of inequality (1.11) in \cite{SY}. The weight $\mu$ can be used to make the inequality sharper. For example, when $M$ is a simply connected surface, we have the following inequality.
\begin{cor}
Let $M^2$ be compact oriented simply connected surface. Then
\begin{equation}\label{eqn-cor-2}
\frac{1}{\sigma_{2}^{(0)}\sigma_{3}^{(0)}\cdots\sigma_{2n}^{(0)}\sigma_{2n+1}^{(0)}}+\frac{1}{2n}\sum_{i=1}^{2n}\left(\frac{1}{\sigma_{1+i}^{(0)}}\right)^{2n}\geq \frac{L(\p M)^{2n}}{2^{2n-1}\pi^{2n} (n!)^2}.
\end{equation}
\end{cor}
\begin{proof}
Let $m=k=2n$, $r=s=1$ , $p=1$, $\mu=\frac{1}{2n}$ and $q=2n$ in \eqref{eqn-thm-2}. Then, the conclusion follows by noting that
\begin{equation}
\lambda_{2i}=\lambda_{2i+1}=\left(\frac{2i\pi}{L(\p M)}\right)^2
\end{equation}
for $i=1,2,\cdots.$
\end{proof}
Note that \eqref{eqn-cor-2} is sharp when $n=1$ and $M$ is the unit disk and is not sharp when $n\geq 2$ on the unit disk.

The strategy to prove Theorem \ref{thm-1} and Theorem \ref{thm-2} in this paper is similar with that in \cite{SY}. The main difference is that we do not apply Courant-Fischer's min-max principle directly to obtain eigenvalue comparison (see Lemma \ref{lem-comp}). This is also a generalization of the key lemma in \cite{YY2}.

The outline of the remaining parts of this paper is as follows. In Section 2, we recall some preliminaries including harmonic conjugate, eigenvalue comparison and matrix inequalities that will be used in Section 3. In Section 3, we prove Theorem \ref{thm-1} and Theorem \ref{thm-2}.
 \section{Preliminaries}
 We first prove an eigenvalue comparison in the same spirit with Courant-Fischer's min-max principle. The result generalizes the key lemma in \cite{YY2}. The proof is similar with that in \cite{YY2}.
\begin{lem}\label{lem-comp}
Let $(M^n,g)$ be a compact oriented Riemannian manifold with nonempty boundary and $$\epsilon_1,\epsilon_2,\cdots,\epsilon_k,\cdots$$
 be a complete orthonormal system of positive Steklov eigenvalues for $p$-forms according to eigenvalues listed in ascending order.  Let $V$ be a finite dimensional subspace of
$$\left\{\omega\in A^p(M)\bigg|\begin{array}{l} \Delta\omega=0,\ \ i_\nu\omega=0,\ \omega\perp_{L^2(\p M)} \mathcal H^{p}_N(M),\\
\mbox{and}\ \omega\perp_{L^2{(\p M)}}\epsilon_1,\epsilon_2,\cdots,\epsilon_{s-1}\end{array}\right\}.$$
Here $p=0,1,2,\cdots,n-1$,  $s$ is a positive integer, $\nu$ is the unit outward normal vector on $\p M$,
\begin{equation}
\mathcal H_N^p(M)=\{\omega\in A^p(M)\ |\ d\omega=\delta\omega=0\ \mbox{and } i_\nu\omega=0\},
\end{equation}
and $\Delta$ is Hodge-Laplacian operator. Suppose that $\dim V=m$.
Then
\begin{equation}
\sigma_{b_p+s+k-1}^{(p)}\leq \lambda_k(A)
\end{equation}
for $k=1,2,\cdots,m$. Here
$$\lambda_1(A)\leq \lambda_2(A)\leq \cdots\leq \lambda_m(A)$$
is the eigenvalues of the linear transform $A:V\to V$ defined by
\begin{equation}
\int_M\left(\vv<d(A\alpha),d\beta>+\vv<\delta (A\alpha),\delta\beta>\right)dV_M=\int_{\p M}\vv<i_\nu d\alpha,i_\nu d\beta>dV_{\p M}
\end{equation}
for any $\alpha,\beta\in V$, and $b_p$ is the $p$-th Betti number of $M$.
\end{lem}
\begin{proof}
Let $\alpha_1,\alpha_2,\cdots,\alpha_m\in V$ be eigenforms of $A$ for $\lambda_1(A),\lambda_2(A),\cdots,\lambda_m(A)$ respectively. By linear algebra, we can also assume that
\begin{equation}\label{eqn-dij}
\int_M\left(\vv<d\alpha_i,d\alpha_j>+\vv<\delta\alpha_i,\delta\alpha_j>\right)dV_M=\delta_{ij}.
\end{equation}
Then
\begin{equation}\label{eqn-aij}
\begin{split}
\int_{\p M}\vv<i_\nu d\alpha_i,i_\nu d\alpha_j>dV_{\p M}=&\lambda_i(A)\int_M\left(\vv<d\alpha_i,d\alpha_j>+\vv<\delta\alpha_i,\delta\alpha_j>\right)dV_M\\
=&\lambda_i(A)\delta_{ij}.
\end{split}
\end{equation}
Let $E_k=\mbox{span}\{\alpha_1,\alpha_2,\cdots,\alpha_k\}$. Then,
$$E_k\cap \overline{\mbox{span}\{\epsilon_{s+k-1},\epsilon_{s+k},\cdots\}}\neq0$$
by dimension reasons.

Let $\omega\in E_k\cap \overline{\mbox{span}\{\epsilon_{s+k-1},\epsilon_{s+k},\cdots\}}$ be nonzero. Suppose
$$\omega=\sum_{i=s+k-1}^\infty c_i\epsilon_i.$$
Then
\begin{equation}
\begin{split}
\frac{\int_{\p M}\vv<i_\nu d\omega,i_\nu d\omega>dV_{\p M}}{\int_M\vv<d\omega,d\omega>+\vv<\delta
\omega,\delta\omega>dV_M}=&\frac{\int_{\p M}\vv<i_\nu d\omega,i_\nu d\omega>dV_{\p M}}{\int_{\p M}\vv<i_\nu d\omega,\omega>dV_M}\\
=&\frac{\sum_{i={s+k-1}}^\infty{\sigma_{b_p+i}^{(p)}}^2c_i^2}{\sum_{i=s+k-1}^\infty\sigma_{b_p+i}^{(p)}c_i^2}\\
\geq&\sigma_{b_p+s+k-1}^{(p)}.
\end{split}
\end{equation}
On the other hand, suppose  $\omega=\sum_{i=1}^kc_i\alpha_i$, by \eqref{eqn-dij} and \eqref{eqn-aij},
\begin{equation}
\frac{\int_{\p M}\vv<i_\nu d\omega,i_\nu d\omega>dV_{\p M}}{\int_M\vv<d\omega,d\omega>+\vv<\delta
\omega,\delta\omega>dV_M}=\frac{\sum_{i=1}^k\lambda_i(A)c_i^2}{\sum_{i=1}^kc_i^2}\leq \lambda_k(A).
\end{equation}
Combing the above two inequalities, we obtain the conclusion.
\end{proof}
Secondly, recall the following result about harmonic conjugate of harmonic functions for higher dimensional manifolds in \cite{YY2,SY}.
\begin{lem}\label{lem-harm-conj}
Let $(M^n,g)$ be a compact oriented Riemannian manifold with nonempty boundary and $u$ be a harmonic function on $M$.  Suppose that
\begin{equation}
*du\perp_{L^2(M)}\mathcal H_N^{(n-1)}(M).
\end{equation}
Then, there is a unique $\omega\in A^{n-2}(M)$ such that
\begin{enumerate}
\item $d\omega=*du;$
\item $\delta\omega=0$;
\item $i_\nu\omega=0$ and
\item $\omega\perp_{L^2(\p M)}\mathcal H_N^{n-2}(M)$.
\end{enumerate}
Here
\begin{equation}
\mathcal H_N^{p}=\{\gamma\in A^p(M)\ |\ d\gamma=\delta\gamma=0\ \mbox{and}\ i_\nu\gamma=0\}.
\end{equation}
$\omega$ is called the harmonic conjugate of $u$.
\end{lem}
Next, recall some matrix inequalities that will be used in the next section. The inequalities is simple and may be well known for experts. However, since we can not find direct reference for them, proofs of them are also given.
\begin{lem}\label{lem-matrix}
Let $A$ be a $m\times m$ matrix that is positive definite and
\begin{equation}
\lambda_1(A)\leq \lambda_2(A)\leq\cdots\leq \lambda_m(A)
\end{equation}
be its eigenvalues. Then,
\begin{enumerate}
\item for any $0\leq p\leq 1$ and $a_1\geq a_2\geq\cdots\geq a_m\geq 0$,
\begin{equation}\label{eqn-A-0-1}
\sum_{i=1}^m(a_i\lambda_i(A))^p\leq \sum_{i=1}^m(a_iA(i,i))^p;
\end{equation}
\item for $p\geq 1$ or $p
\leq 0$, and $0< a_1\leq a_2\leq\cdots\leq a_m$,
\begin{equation}\label{eqn-A-1-1}
\sum_{i=1}^m(a_i\lambda_i(A))^p\geq \sum_{i=1}^m(a_iA(i,i))^p;
\end{equation}
\item for any $p\leq 0$ and $k=1,2,\cdots,m$
\begin{equation}\label{eqn-A-0-0}
\begin{split}
&\sum_{1\leq i_1<i_2<\cdots< i_k\leq m}\left(\lambda_{i_1}(A)\lambda_{i_2}(A)\cdots\lambda_{i_k}(A)\right)^p\\
\geq& \sum_{1\leq i_1<i_2<\cdots< i_k\leq m}\left(A(i_1,i_1)A(i_2,i_2)\cdots A(i_k,i_k)\right)^p.
\end{split}
\end{equation}
\end{enumerate}
Here $A(i,j)$ means the $(i,j)$-entry of $A$.
\end{lem}
\begin{proof}
\begin{enumerate}
\item By Schur's Theorem (see \cite{Bh}),
\begin{equation}
\{A(1,1),A(2,2),\cdots,A(m,m)\}\prec\{\lambda_1(A),\lambda_2(A),\cdots,\lambda_m(A)\}.
\end{equation}
 Note that, $f(t)=-t^p$ is convex for $0\leq p\leq 1$. By basic majorization principles,  we have
\begin{equation}\label{eqn-A-1}
\begin{split}
&\{-A(1,1)^p,-A(2,2)^p,\cdots,-A(m,m)^p\}\\
\prec_w&\{-\lambda_1(A)^p,-\lambda_2(A)^p,\cdots,-\lambda_m(A)^p\}.
\end{split}
\end{equation}
Let $\sigma:\{1,2,\cdots,m\}\to\{1,2,\cdots,m\}$ be a permutation such that
\begin{equation*}
A(\sigma(1),\sigma(1))\leq A(\sigma(2),\sigma(2))\leq \cdots \leq A(\sigma(m),\sigma(m)).
\end{equation*}
Then, by \eqref{eqn-A-1} and rearrangement inequality,
\begin{equation}\label{eqn-abel-1}
\begin{split}
&-a_1^p \lambda_1^p-a_2^p \lambda_2^p-\cdots-a_m^p\lambda_m^p\\
=& -(a_1^p-a_2^p)\lambda_1^p-(a_2^p-a_3^p)(\lambda_1^p+\lambda_2^p)-\cdots-(a_{m-1}^p-a_m^p)(\lambda_1^p+\lambda_2^p+\cdots+\lambda_{m-1}^p)\\
&-a_m^p(\lambda_1^p+\lambda_2^p+\cdots+\lambda_{m}^p)\\
\geq&-(a_1^p-a_2^p)A(\sigma(1),\sigma(1))^p-(a_2^p-a_3^p)(A(\sigma(1),\sigma(1))^p+A(\sigma(2),\sigma(2))^p)-\cdots\\
&-(a_{m-1}^p-a_m^p)(A(\sigma(1),\sigma(1))^p+A(\sigma(2),\sigma(2))^p+\cdots+A(\sigma(m-1),\sigma(m-1))^p)\\
&-a_m^p(A(\sigma(1),\sigma(1))^p+A(\sigma(2),\sigma(2))^p+\cdots+A(\sigma(m),\sigma(m))^p)\\
=&-a_1^pA(\sigma(1),\sigma(1))^p-a_2^pA(\sigma(2),\sigma(2))^p-\cdots-a_m^pA(\sigma(m),\sigma(m))^p\\
\geq&-a_1^pA(1,1)^p-a_2^pA(2,2)^p-\cdots-a_m^pA(m,m)^p.\\
\end{split}
\end{equation}
This completes the proof of (1).
\item Note that $f(t)=t^p$ is convex for $p\geq 1$ or $p\leq 0$. So, by basic majorization principles ,
\begin{equation*}
\begin{split}
&\{A(1,1)^p,A(2,2)^p,\cdots,A(m,m)^p\}\\
\prec_w&\{\lambda_1(A)^p,\lambda_2(A)^p,\cdots,\lambda_m(A)^p\}.
\end{split}
\end{equation*}
Then, a similar argument as in \eqref{eqn-abel-1} will give us (2).
\item
 Applying (2) to the exterior power $\wedge^kA$ of $A$ (see \cite{Bh}) with all $a_i$'s being 1 , we have
\begin{equation}
\sum_{1\leq i_1<i_2<\cdots< i_r\leq m}\left(\lambda_{i_1}(A)\lambda_{i_2}(A)\cdots\lambda_{i_r}(A)\right)^p\geq \sum_{1\leq i_1<i_2<\cdots< i_r\leq m}\left(A\left[\begin{array}{l}i_1,i_2,\cdots,i_r\\
i_1,i_2,\cdots,i_r
\end{array}\right]\right)^p.
\end{equation}
Here
\begin{equation}
A\left[\begin{array}{l}i_1,i_2,\cdots,i_r\\
i_1,i_2,\cdots,i_r
\end{array}\right]=\left|\begin{array}{llll}A(i_1,i_1)&A(i_1,i_2)&\cdots&A(i_1,i_r)\\
A(i_2,i_1)&A(i_2,i_2)&\cdots&A(i_2,i_r)\\
\vdots&\vdots&\cdots&\vdots\\
A(i_r,i_1)&A(i_r,i_2)&\cdots&A(i_r,i_r)\\
\end{array}\right|.
\end{equation}
By Hadamard's inequality,
\begin{equation}
 A\left[\begin{array}{l}i_1,i_2,\cdots,i_r\\
i_1,i_2,\cdots,i_r
\end{array}\right]\leq A(i_1,i_1)A(i_2,i_2)\cdots A(i_r,i_r).
\end{equation}
So, we have (3).
\end{enumerate}
\end{proof}
Finally, recall the following elementary inequality that will be used in the next section. For completeness, we also give a proof.
\begin{lem}\label{lem-young}
Let $x_1,x_2,\cdots,x_m$ and $p_1,p_2,\cdots,p_m$ be positive numbers. Then
\begin{equation}
x_1^{p_1}+x_2^{p_2}+\cdots+x_m^{p_m}\geq \frac{1}{p} \left(p_1^\frac{1}{p_1}p_2^\frac{1}{p_2}\cdots p_m^\frac{1}{p_m}\right)^p( x_1x_2\cdots x_m)^p.
\end{equation}
Here
\begin{equation}
\frac{1}{p}=\frac{1}{p_1}+\frac{1}{p_2}+\cdots+\frac{1}{p_m}.
\end{equation}
\begin{proof}
Let $q_i=\frac{p_i}{p}$ for $i=1,2,\cdots,m$. Then
\begin{equation}
\frac{1}{q_1}+\frac{1}{q_2}+\cdots+\frac{1}{q_m}=1.
\end{equation}
By Young's inequality,
\begin{equation}
\begin{split}
&x_1^{p_1}+x_2^{p_2}+\cdots+x_m^{p_m}\\
=&\frac{1}{q_1}\left(q_1^\frac{1}{q_1}x_1^p\right)^{q_1}+\frac{1}{q_2}\left(q_2^\frac{1}{q_2}x_2^p\right)^{q_2}+\cdots\frac{1}{q_m}\left(q_m^\frac{1}{q_m}x_m^p\right)^{q_m}\\
\geq& q_1^\frac{1}{q_1}q_2^\frac{1}{q_2}\cdots q_m^\frac{1}{q_m}(x_1x_2\cdots x_m)^p\\
=&\frac{1}{p} \left(p_1^\frac{1}{p_1}p_2^\frac{1}{p_2}\cdots p_m^\frac{1}{p_m}\right)^p( x_1x_2\cdots x_m)^p.
\end{split}
\end{equation}
This completes the proof of the inequality.
\end{proof}
\end{lem}
\section{Proof of the main theorems}
In this section, we prove Theorem \ref{thm-1} and Theorem \ref{thm-2}. First, by using Lemma \ref{lem-comp} and Lemma \ref{lem-harm-conj}, we have the following comparison of eigenvalues.
\begin{lem}\label{lem-comp-e}
Let $(M^n,g)$ be a compact oriented Riemannian manifold with nonempty boundary. Then, for any positive integers $r,s$ and $m$, there are two $m\times m$ matrices $A$ and $B$ that are both positive definite such that
\begin{enumerate}
\item $\sigma_{r+i}^{(0)}\leq \lambda_i(A)$;
\item $\sigma_{b_{n-2}+s+i-1}^{(n-2)}\leq \lambda_i(B)$ and
\item $B(i,i)\leq A^{-1}(i,i)\lambda_{b_{n-1}+r+s+i-1}$,
\end{enumerate}
for $i=1,2,\cdots,m$. Here $A^{-1}(i,j)$ and $B(i,j)$ mean the $(i,j)$-entry of $A^{-1}$ and $B$ respectively, $\lambda_k$ means the $k$-th eigenvalue for the Laplacian operator on $\p M$, and $b_k$ means the $k$-th Betti number of $M$.
\end{lem}
\begin{proof}
Let
$$\phi_1=\frac{1}{\sqrt{A(\p M)}}, \phi_2,\cdots,\phi_k,\cdots$$
be a complete orthonormal system for eigenvalues of the Laplacian operator of $\p M$ according to eigenvalues
$$0=\lambda_1\leq \lambda_2\leq\cdots\leq\lambda_k\leq\cdots.$$
Here $A(\p M)$ means the area of $\p M$. Moreover, let
$$\psi_1,\psi_2,\cdots,\psi_k,\cdots$$
and
$$\epsilon_1,\epsilon_2,\cdots,\epsilon_k,\cdots$$
be complete orhtonormal systems for positive Steklov eigenvalues of functions and $(n-2)$-forms repectively, according to eigenvalues listed in ascending order.

By the same argument as in the proof of Theorem 1.1 in \cite{SY}, there are
nonconstant harmonic functions $u_1,u_2,\cdots,u_m$ such that
\begin{enumerate}
\item $*du_i\perp_{L^2(M)}\mathcal H^{n-1}_N(M)$;
\item $u_i\perp_{L^2(\p M)}\psi_1,\psi_2,\cdots,\psi_{r-1}$;
\item $\omega_i\perp_{L^2(\p M)}\epsilon_1,\epsilon_2,\cdots,\epsilon_{s-1}$ where $\omega_i$ is the harmonic conjugate of $u_i$ as in Lemma \ref{lem-harm-conj};
\item $u_i\in \mbox{span}\{\hat \phi_2,\hat\phi_3,\cdots, \hat \phi_{b_{n-1}+r+s+i-1}\}$ where $\hat\phi_i$ means the harmonic extension of $\phi_i$;
\item $\int_M\vv<du_i,du_j>dV_M=\delta_{ij}$
\end{enumerate}
for $i,j=1,2,\cdots,m$. For making the argument more self-contained, we sketch the construction of $u_1,u_2,\cdots, u_m$ in the following. Suppose that $u_1,u_2,\cdots, u_{k-1}$ satisfying (1),(2),(3),(4) and
\begin{equation}
\int_M\vv<du_i,du_j>dV_M=\delta_{ij}\ \mbox{for}\ i,j=1,2,\cdots,k-1,
\end{equation}
has been constructed. Suppose that
\begin{equation}
u_k=c_2\hat\phi_2+c_3\hat\phi_3+\cdots+c_{b_{n-1}+r+s+k-1}\hat\phi_{b_{n-1}+r+s+k-1}
\end{equation}
with $c_i$'s constants to be determined. Note that (1),(2),(3) and
\begin{equation}\label{eqn-uk}
\int_M\vv<du_k,du_i>dV_M=0\ \mbox{for}\ i=1,2,\cdots,k-1
\end{equation}
make
$$b_{n-1}+(r-1)+s-1+k-1=b_{n-1}+r+s+k-3$$
homogeneous linear restrictions on the $b_{n-1}+r+s+k-2$ unknowns $c_2,c_3,\cdots,c_{b_{n-1}+r+s+k-1}$. Because the number of unknowns is greater than the number of homogeneous linear restrictions, there is a nonconstant $u_k$ satisfying (1),(2),(3),(4) and \eqref{eqn-uk}. By re-scale $u_k$, we can suppose that
\begin{equation}
\int_M\vv<du_k,du_k>dV_M=1.
\end{equation}
This give us the construction of $u_1,u_2,\cdots,u_m$.

Note that
\begin{equation}
\int_{M}\vv<d\omega_i,d\omega_j>dV_M=\int_{M}\vv<du_i,du_j>dV_M=\delta_{ij}
\end{equation}
for $i,j=1,2,\cdots,m$.

Let $V=\mbox{span}\{u_1,u_2,\cdots,u_m\}$ and $W=\mbox{span}\{\omega_1,\omega_2,\cdots,\omega_m\}$. Let $A:V\to V$ and $B:W\to W$
be linear transformation on $V$ and $W$ such that
\begin{equation}\label{eqn-def-A}
\int_M\vv<du,dv>dV_M=\int_{\p M}\vv<Au,v>dV_{\p M}
\end{equation}
for any $u,v\in V$ and
\begin{equation}\label{eqn-def-B}
\int_M\vv<d B\alpha,d\beta>=\int_{\p M}\vv<i_\nu\alpha,i_\nu\beta>dV_{\p M}
\end{equation}
for any $\alpha,\beta\in W$ respectively. Then, by Courant-Fischer's min-max principle,
\begin{equation}\label{eqn-s-A}
\sigma_{r+i}^{(0)}\leq \lambda_i(A),
\end{equation}
and by Lemma \ref{lem-comp},
\begin{equation}\label{eqn-s-B}
\sigma_{b_{n-2}+s+i-1}^{(n-2)}\leq \lambda_i(B)
\end{equation}
for $i=1,2,\cdots,m$.

Denote the matrix of $A$ and $B$ under the basis $\{u_1,u_2,\cdots,u_m\}$ and $\{\omega_1,\omega_2,\cdots,\omega_m\}$ as $A$ and $B$ respectively. Then, by \eqref{eqn-def-A} and \eqref{eqn-def-B},
\begin{equation}
A^{-1}(i,j)=\int_{\p M}\vv<u_i,u_j>dV_{\p M}\ \mbox{and}\ B(i,j)=\int_{\p M}\vv<i_\nu\omega_i,i_\nu\omega_j>dV_{\p M},
\end{equation}
for $i,j=1,2,\cdots,m$. Moreover,
\begin{equation}\label{eqn-AB}
\begin{split}
\frac{B(i,i)}{A^{-1}(i,i)}=&\frac{\int_{\p M}\vv<i_\nu\omega_{i},i_\nu\omega_{i}>dV_{\p M}}{\int_{\p M}\vv<u_{i},u_{i}>dV_{\p M}}\\
=&\frac{\int_{\p M}\vv<i_\nu*du_i,i_\nu*du_{i}>dV_{\p M}}{\int_{\p M}\vv<u_{i},u_{i}>dV_{\p M}}\\
=&\frac{\int_{\p M}\vv<du_i,du_{i}>dV_{\p M}}{\int_{\p M}\vv<u_{i},u_{i}>dV_{\p M}}\\
\leq& \lambda_{b_{n-1}+r+s+i-1}.
\end{split}
\end{equation}
This completes the proof of Lemma \ref{lem-comp-e}.
\end{proof}
\begin{rem}
\eqref{eqn-s-A} can also be shown by similar arguments as in the proof of \eqref{lem-comp}.
\end{rem}
Now, we are ready to prove Theorem \ref{thm-1} and Theorem \ref{thm-2}.
\begin{proof}[Proof of Theorem \ref{thm-1}]Let $A,B$ be the matrices in Lemma \ref{lem-comp-e}. Then, by Lemma \ref{lem-comp-e}, \eqref{eqn-A-0-1} and \eqref{eqn-A-1-1}, we have
\begin{equation}
\begin{split}
\sum_{i=1}^m\left(a_i\sigma_{b_{n-2}+s+i-1}^{(n-2)}\right)^\frac{1}{p}\leq& \sum_{i=1}^m\left(a_i\lambda_i(B)\right)^\frac{1}{p}\\
\leq& \sum_{i=1}^m\left(a_iB(i,i)\right)^\frac{1}{p}\\
\leq&\sum_{i=1}^mA^{-1}(i,i)^\frac{1}{p}\left(a_i\lambda_{b_{n-1}+r+s+i-1}\right)^\frac{1}{p}\\
=&\sum_{i=1}^m\left(c_iA^{-1}(i,i)\right)^\frac{1}{p}\left(a_ic_i^{-1}\lambda_{b_{n-1}+r+s+i-1}\right)^\frac{1}{p}\\
\leq& \left(\sum_{i=1}^{m}\left(c_iA^{-1}(i,i)\right)^{\frac{q}{p}}\right)^\frac{1}{q}\left(\sum_{i=1}^m\left(\frac{a_i\lambda_{b_{n-1}+r+s+i-1}}{c_i}\right)^\frac{q^*}{p}\right)^{\frac{1}{q^*}}\\
\leq&\left(\sum_{i=1}^{m}\left(\frac{c_i}{\lambda_i(A)}\right)^{\frac{q}{p}}\right)^\frac{1}{q}\left(\sum_{i=1}^m\left(\frac{a_i\lambda_{b_{n-1}+r+s+i-1}}{c_i}\right)^\frac{q^*}{p}\right)^{\frac{1}{q^*}}\\
\leq&\left(\sum_{i=1}^{m}\left(\frac{c_i}{\sigma_{r+i}^{(0)}}\right)^{\frac{q}{p}}\right)^\frac{1}{q}\left(\sum_{i=1}^m\left(\frac{a_i\lambda_{b_{n-1}+r+s+i-1}}{c_i}\right)^\frac{q^*}{p}\right)^{\frac{1}{q^*}}.
\end{split}
\end{equation}
This completes the proof of Theorem \ref{thm-1}.
\end{proof}
Similarly as in the proof of Theorem \ref{thm-1}, by using Lemma \ref{lem-comp-e} and Lemma \ref{lem-matrix}, we can prove Theorem \ref{thm-2}.
\begin{proof}[Proof of Theorem \ref{thm-2}]
Let $A$ and $B$ be the matrices in Lemma \ref{lem-comp-e}. Then, by Lemma \ref{lem-comp-e}, \eqref{eqn-A-0-0}, \eqref{eqn-A-1-1} and Lemma \ref{lem-young},
\begin{equation}
\begin{split}
&\sum_{1\leq i_1<i_2<\cdots<i_k\leq m}\left(\frac{1}{\sigma_{b_{n-2}+s+i_1-1}^{(n-2)}\sigma_{b_{n-2}+s+i_2-1}^{(n-2)}\cdots\sigma_{b_{n-2}+s+i_k-1}^{(n-2)}}\right)^p+\mu C_{m-1}^{k-1}\sum_{i=1}^{m}\left(\frac{1}{\sigma_{r+i}^{(0)}}\right)^q\\
\geq&\sum_{1\leq i_1<i_2<\cdots<i_k\leq m}\left(\frac{1}{\lambda_{i_1}(B)\lambda_{i_2}(B)\cdots\lambda_{i_k}(B)}\right)^p+\mu C_{m-1}^{k-1}\sum_{i=1}^{m}\left(\frac{1}{\lambda_i(A)}\right)^q\\
\geq&\sum_{1\leq i_1<i_2<\cdots<i_k\leq m}\left(\frac{1}{B(i_1,i_1)B(i_2,i_2)\cdots B(i_k,i_k)}\right)^p+\mu C_{m-1}^{k-1}\sum_{i=1}^{m}\left(A^{-1}(i,i)\right)^q\\
=&\sum_{1\leq i_1<i_2<\cdots<i_k\leq m}\bigg[\left(\frac{1}{B(i_1,i_1)B(i_2,i_2)\cdots B(i_k,i_k)}\right)^p\\
&+\mu\left(A^{-1}(i_1,i_1)\right)^q+\mu\left(A^{-1}(i_2,i_2)\right)^q+\cdots+\mu\left(A^{-1}(i_k,i_k)\right)^q\bigg]\\
\geq&\frac{kp+q}{pq}p^\frac{q}{kp+q}q^\frac{kp}{kp+q}\mu^\frac{kp}{k+q}\sum_{1\leq i_1<i_2<\cdots<i_k\leq m}\left(\frac{A^{-1}(i_1,i_1)A^{-1}(i_2,i_2)\cdots A^{-1}(i_k,i_k)}{B(i_1,i_1)B(i_2,i_2)\cdots B(i_k,i_k)}\right)^\frac{pq}{kp+q}\\
\geq&\frac{kp+q}{pq}p^\frac{q}{kp+q}q^\frac{kp}{kp+q}\mu^\frac{kp}{kp+q}\times\\
&\sum_{1\leq i_1<i_2<\cdots<i_k\leq m}\left(\frac{1}{\lambda_{b_{n-1}+r+s+i_1-1}\lambda_{b_{n-1}+r+s+i_2-1}\cdots\lambda_{b_{n-1}+r+s+i_k-1}}\right)^\frac{pq}{kp+q}.\\
\end{split}
\end{equation}
This completes the proof of Theorem \ref{thm-2}.
\end{proof}


\begin{thebibliography}{99}
\bibitem{Bh} Bhatia, Rajendra {\it Matrix analysis}. Graduate Texts in Mathematics, 169. Springer-Verlag, New York, 1997. xii+347 pp. ISBN: 0-387-94846-5.
\bibitem{Br}Brock, F. {\it An isoperimetric inequality for eigenvalues of the Stekloff problem.}  Z. Angew. Math. Mech. 81 (2001), no. 1, 69--71.
    \bibitem{CEG}Colbois, B.; El Soufi, A.; Girouard, A., {\sl Isoperimetric control of the Steklov spectrum.} J. Funct. Anal. 261 (2011), no. 5, 1384--1399.
\bibitem{Di}Dittmar, Bodo, {\it Sums of reciprocal Stekloff eigenvalues.} Math. Nachr. 268 (2004), 44--49.
\bibitem{E}Escobar, J., {\it A comparison theorem for the first non-zero Steklov eigenvalue.} J. Funct. Anal. 178 (2000), no. 1, 143--155.
\bibitem{FS1} Fraser, A.; Schoen, R., {\sl
The first Steklov eigenvalue, conformal geometry, and minimal surfaces,} Adv. Math. \textbf{226} (2011), no. 5, 4011--4030.

 \bibitem{FS2} Fraser, A.; Schoen, R., {\sl Sharp eigenvalue bounds and minimal surfaces in the ball}, Invent. Math., to appear.
\bibitem{GP}Girouard, A.; Polterovich, I. {\it Spectral geometry of the Steklov problem.} Journal of Spectral Theory, to appear.
\bibitem{GP1}Girouard, A.; Polterovich, I. {\it Upper bounds for Steklov eigenvalues on surfaces.} Electron. Res. Announc. Math. Sci. 19 (2012), 77--85.
\bibitem{GP2}Girouard, Alexandre; Polterovich, Iosif {\it Shape optimization for low Neumann and Steklov eigenvalues.} Math. Methods Appl. Sci. 33 (2010), no. 4, 501--516.
     \bibitem{GP3}Girouard, A.; Polterovich, I. {\it On the Hersch-Payne-Schiffer estimates for the eigenvalues of the Steklov problem.} (Russian) Funktsional. Anal. i Prilozhen. 44 (2010), no. 2, 33--47; translation in Funct. Anal. Appl. 44 (2010), no. 2, 10--117.
\bibitem{HPS} Hersch, J.; Payne, L. E.; Schiffer, M. M. {\it Some inequalities for Stekloff eigenvalues.} Arch. Rational Mech. Anal. 57 (1975), 99--114.
\bibitem{IM}Ilias, S.; Makhoul, O.{\it A Reilly inequality for the first Steklov eigenvalue.} Differential Geom. Appl. 29 (2011), no. 5, 699--708.
\bibitem{Ka}Karpukhin, Mikhail A. {\it Bounds between Laplace and Steklov eigenvalues on nonnegatively curved manifolds.} arXiv:1512.09038.

\bibitem{Ku}Kuznetsov, Nikolay; Kulczycki, Tadeusz; Kwa\'snicki, Mateusz; Nazarov, Alexander; Poborchi, Sergey; Polterovich, Iosif; Siudeja, Bart{\l}omiej. {\it The legacy of Vladimir Andreevich Steklov.} Notices Amer. Math. Soc. 61 (2014), no. 1, 9--22.
\bibitem{K}Kwong, Kwok-Kun, {\it Some sharp Hodge Laplacian and Steklov eigenvalue estimates for differential forms.} Calc. Var. Partial Differential Equations 55 (2016), no. 2, Art. 38, 14 pp.
\bibitem{RS1}Raulot, S.; Savo, A. {\it On the spectrum of the Dirichlet-to-Neumann operator acting on forms of a Euclidean domain.} J. Geom. Phys. 77 (2014), 1--12.
\bibitem{RS2}Raulot, S.; Savo, A. {\it On the first eigenvalue of the Dirichlet-to-Neumann operator on forms.} J. Funct. Anal. 262 (2012), no. 3, 889--914.
\bibitem{RS3}Raulot, S.; Savo, A. {\it A Reilly formula and eigenvalue estimates for differential forms.} J. Geom. Anal. 21 (2011), no. 3, 620--640.
\bibitem{S}Schwarz, G. {\it Hodge decomposition¡ªa method for solving boundary value problems.} Lecture Notes in Mathematics, 1607. Springer-Verlag, Berlin, 1995. viii+155 pp. ISBN: 3-540-60016-7.
\bibitem{SY}Shi, Yongjie; Yu, Chengjie,{\it Trace and inverse trace of Steklov eigenvalues.} J. Differential Equations 261 (2016), no. 3, 2026--2040.
\bibitem{St}Stekloff, W., {\sl Sur les probl\`emes fondamentaux de la physique math\'ematique.}   Ann. Sci. \'Ecole Norm. Sup. (3) 19 (1902), 191--259.
\bibitem{Ta}Taylor, Michael E. {\it Partial differential equations II. Qualitative studies of linear equations.} Second edition. Applied Mathematical Sciences, 116. Springer, New York, 2011. xxii+614 pp. ISBN: 978-1-4419-7051-0.
\bibitem{Ul}Ulmann, G. {\it Electrical impedance tomography and Calder¡äon¡¯s
problem},http://www.math.washington.edu/~gunther/publications
/Papers/calderoniprevised.pdf
\bibitem{WX}Wang, Qiaoling; Xia, Changyu {\it Sharp bounds for the first non-zero Stekloff eigenvalues.} J. Funct. Anal. 257 (2009), no. 8, 2635--2644.
\bibitem{W}Weinstock, R., {\it Inequalities for a classical eigenvalue problem.} J. Rational Mech. Anal. 3, (1954).
\bibitem{YY}Yang, Liangwei; Yu, Chengjie, {\it Estimates for higher Steklov eigenvalues}, arXiv:1601.01882.
\bibitem{YY2}Yang, Liangwei; Yu, Chengjie, {\it A higher dimensional generalization of the Hersch-Payne-Schiffer inequality for Steklov eigenvalues}, arXiv:1508.06026.

\end{thebibliography}
\end{document}